\documentclass[12pt, reqno]{amsart}

\title{On maximal families of independent sets with respect to asymptotic density}

\usepackage[T1]{fontenc}
\usepackage{amsmath}
\usepackage{amssymb}
\usepackage{amsthm}
\usepackage[left=3.5cm, right=3.5cm]{geometry} %ORIGINAL
\usepackage{hyperref}
\usepackage{fancyhdr}
\usepackage{enumitem}
\usepackage{comment}
\usepackage{nicefrac}
\usepackage{comment}
\usepackage{bm}
\usepackage{mathrsfs}
\usepackage{graphicx}
\usepackage[utf8]{inputenc}
\usepackage{cancel}
\usepackage{mathtools}

\usepackage{tikz}
\usetikzlibrary{decorations.pathreplacing,matrix,arrows,positioning,automata,shapes,shadows,calc,fadings,decorations,snakes,through,intersections}

\AtBeginDocument{%
   \def\MR#1{}
}

\newtheorem{thm}{Theorem}[section]
\newtheorem{cor}[thm]{Corollary}%[section]

\newtheorem{prop}[thm]{Proposition}

\theoremstyle{definition} 
\newtheorem{defi}[thm]{Definition}%[section]
\let\olddefi\defi
\renewcommand{\defi}{\olddefi\normalfont}
\newtheorem{question}{Question}%[section]
\let\oldquestion\question
\renewcommand{\question}{\oldquestion\normalfont}
\newtheorem{example}[thm]{Example}
\let\oldexample\example
\renewcommand{\example}{\oldexample\normalfont}
\newtheorem{rmk}[thm]{Remark}
\let\oldrmk\rmk
\renewcommand{\rmk}{\oldrmk\normalfont}
\newtheorem{claim}{\textsc{Claim}}

\renewcommand{\ast}{\star}

{\textsc{J. M. Keith} and \textsc{P. Leonetti}}
\hypersetup{
    pdftitle={},
    pdfauthor={},
    pdfmenubar=false,
    pdffitwindow=true,
    pdfstartview=FitH,
    colorlinks=true,
    linkcolor=blue,      %%%%%
    citecolor=green,
}

\providecommand{\MR}[1]{}

\providecommand{\MR}{\relax\ifhmode\unskip\space\fi MR }

\author[J.~M.~Keith]{Jonathan M. Keith}
\address{School of Mathematics, Monash University, Wellington Rd, Clayton VIC 3800, Australia}
\email{jonathan.keith@monash.edu}

\author[P.~Leonetti]{Paolo Leonetti}
\address{Department of Economics, Universit\`a degli Studi dell'Insubria, via Monte Generoso 71, 21100 Varese, Italy}
\email{leonetti.paolo@gmail.com}


\begin{document}

\maketitle
\thispagestyle{empty}

\begin{abstract}
    \noindent 
    We study families of subsets of $\omega$ which are independent with respect to the asymptotic density $\mathsf{d}$.
    We show, for instance, that there exists a maximal $\mathsf{d}$-independent family $\mathcal{A}$ such that $\mathsf{d}[\mathcal{A}]$ attains a prescribed set of values in $(0,1)$ with at most countably many exceptions. In addition, under $\mathrm{cov}(\mathcal{N})=\mathfrak{c}$, it is possible to construct such $\mathcal{A}$ with no exceptions. 
    We also construct $2^{\mathfrak{c}}$ maximal $\mathsf{d}$-independent families with pairwise distinct generated density fields and obtain maximal families with strong definability pathologies, including examples without the Baire property and, consistently, nonmeasurable examples. 
\end{abstract}

\section{Introduction}\label{sec:intro}

Let $\mathcal{D}$ be the family of sets of natural numbers which admit asymptotic density, that is, the family of sets $A\subseteq \omega:=\{0,1,2,\ldots\}$ for which the limit 
$$
\mathsf{d}(A):=\lim_{n\to \infty}\frac{|A\cap n|}{n}
$$
exists. (As usual we identify $n\in \omega$ with $\{0,1,\ldots,n-1\}$, so that $|A\cap n|$ counts the elements of $A$ which are $<n$.) In other words, $\mathcal{D}$ represents the domain $\mathrm{dom}(\mathsf{d})$ of the asymptotic density map $\mathsf{d}$, see e.g. \cite{MR1845008, MR4054777}. It is folklore that $\mathcal{D}$ is stable under complements and finite disjoint unions, but not under finite unions (or, equivalently, finite intersections); in particular, $\mathcal{D}$ is not a field of sets. 

In this work, we introduce and study the notion of [maximal] $\mathsf{d}$-independent family $\mathcal{A}\subseteq \mathcal{D}$. Roughly, $\mathcal{A}$ will represent a family of \textquotedblleft independent events\textquotedblright\, on the space $\mathcal{D}$ with respect to the finitely additive map $\mathsf{d}$. 
Hereafter, we denote by $\langle \mathcal{A}\rangle$ the field of sets generated by $\mathcal{A}$.  
For each $A\subseteq \omega$ and $n \in \omega$, we write $A^1:=A$ and $A^0:=\omega\setminus A$, and $n\cdot A:=\{na: a \in A\}$.
Our results will be given within $\mathsf{ZFC}$; additional set theory axioms will be always explicit. 
\begin{defi}\label{def:independent}
    A nonempty family $\mathcal{A}\subseteq \mathcal{D}$ is said to be $\mathsf{d}$-\emph{independent} if: 
    \begin{enumerate}[label={\rm (\textsc{d}\arabic*)}]
        \item \label{defi:1indep} $\mathsf{d}(A) \in (0,1)$ for all $A \in \mathcal{A}$;
        \item \label{defi:2indep} $\langle \mathcal{A}\rangle\subseteq \mathcal{D}$; 
        \item \label{defi:3indep} 
        $\mathsf{d}(\bigcap \mathcal{F}) = \prod_{A \in \mathcal{F}}\mathsf{d}(A)$ for all nonempty finite subfamily $\mathcal{F} \subseteq \mathcal{A}$. 
    \end{enumerate}
\end{defi}
A trivial example of a $\mathsf{d}$-independent family is $\mathcal{A}=\{p\cdot \omega: p \text{ prime}\}$. 
A variant of the above definition for arbitrary subfamilies of $\mathcal{P}(\omega)$, which has been termed \textquotedblleft *-independence,\textquotedblright\, appeared in \cite[Definition 3.2]{MR4609469} and \cite[Definition 1.5]{Valderrama}. 
It is worth remarking that, if $|\mathcal{A}|\ge 2$, item \ref{defi:2indep} cannot be replaced by: 
\begin{enumerate}[label={\rm (\textsc{d}\arabic*$^\natural$)}]
\setcounter{enumi}{1}
\item \label{defi:2indepflat} $\mathcal{A}$ is stable under finite intersections.
\end{enumerate}
In fact, if $A,B \in \mathcal{A}$ are distinct then $C:=A\cap B \in \mathcal{A}$ and $\mathsf{d}(C)=\mathsf{d}(A\cap C)=\mathsf{d}(A)\mathsf{d}(C)$. This would contradict item \ref{defi:1indep}. 

At the same time, observe that, if $\mathcal{A}=\{A_\alpha: \alpha<\kappa\}$ is a $\mathsf{d}$-independent family, then item \ref{defi:3indep} can be rewritten as $\mathsf{d}(\bigcap_{\alpha \in F} A_\alpha^1) = \prod_{\alpha \in F}\mathsf{d}(A_\alpha^1)$ for each nonempty finite $F\subseteq \kappa$. However, since $\langle \mathcal{A}\rangle \subseteq \mathcal{D}$, then also finite intersections of sets of the type $A_\alpha^1$ or $A_\alpha^0$ belong to $\mathcal{D}$. 
\begin{comment}
It is straightforward to obtain that, if $F\subseteq \kappa$ is nonempty finite, then
\begin{equation}\label{eq:d3sharp}
\forall \sigma \in 2^F, \qquad 
\mathsf{d}\left(\bigcap_{\alpha \in F} A_\alpha^{\sigma(\alpha)}\right) 
=\prod_{\alpha \in F}\mathsf{d}\left(A_\alpha^{\sigma(\alpha)}\right). 
\end{equation} 
Of course, \eqref{eq:d3sharp} corresponds to item \ref{defi:3indep} if $\sigma=1$. 
It also follows that each $\bigcap_{\alpha \in F} A_\alpha^{\sigma(\alpha)}$ is necessarily infinite, hence $\mathcal{A}$ is a (classical) independent family of subsets of $\omega$, see e.g. \cite{MR2768685, MR4510832}. 
%
Our main results are given in Section \ref{sec:mainresults}. The proofs are presented in Section \ref{sec:proofs}. 
\end{comment}
It is straightforward to obtain that, if $F\subseteq \kappa$ is nonempty finite, then
\begin{equation}\label{eq:d3sharp}
\forall \sigma \in 2^F, \qquad 
\mathsf{d}\left(\bigcap_{\alpha \in F} A_\alpha^{\sigma(\alpha)}\right) 
=\prod_{\alpha \in F}\mathsf{d}\left(A_\alpha^{\sigma(\alpha)}\right). 
\end{equation} 
Of course, \eqref{eq:d3sharp} corresponds to item \ref{defi:3indep} if $\sigma=1$. 
\begin{rmk}\label{rmk:relationshipindependent}
    It also follows from \eqref{eq:d3sharp} that, if $\mathcal{A}=\{A_\alpha: \alpha<\kappa\}$ is a $\mathsf{d}$-independent family, then each $\bigcap_{\alpha \in F} A_\alpha^{\sigma(\alpha)}$ is necessarily infinite, that is, $\mathcal{A}$ is a (classical) independent family of subsets of $\omega$, see e.g. \cite{MR2768685, MR4510832}. 

    Vice versa, if $\mathcal{B}=\{B_\alpha:\alpha<\kappa\}$ is an independent family on $\omega$, then it is possible to construct a $\mathsf{d}$-independent family $\mathcal{A}=\{A_\alpha: \alpha<\kappa\}$ as follows. 
    Choose pairwise disjoint consecutive finite intervals $(I_m: m \in\omega)$ such that $|I_m|=2^m(m+1)!$ for all $m \in\omega$. For each $m \in\omega$, partition each $I_m$ into $2^m$ subsets $\{I_{m,\sigma}: \sigma\in 2^m\}$ such that each $I_{m,\sigma}$ is a translation of $I_m\cap (2^m\cdot \omega)$. Then, define 
    $$
    \forall \alpha<\kappa, \quad \quad 
    A_\alpha:=\bigcup_{m \in \omega}\bigcup_{\sigma \in S_{\alpha,m}}I_{m,\sigma},
    $$
    where $S_{\alpha,m}:=\left\{\sigma \in 2^m: \sum_{i \in B_\alpha \cap m}\sigma(i) \text{ is odd}\right\}$. Then $\mathcal{A}=\{A_\alpha: \alpha<\kappa\}$ is a $\mathsf{d}$-independent family. Indeed, fix distinct $\alpha_0,\ldots,\alpha_{k-1}<\kappa$. Since every nontrivial 
    intersection of finitely many of the $B_{\alpha}$'s and their complements 
    is infinite, then $B_{\alpha_j}\setminus \bigcup_{t \in k, t\neq j}B_{\alpha_t}$ is nonempty for each $j \in k$. Hence  
    there is $m_0\in \omega$ such that for all $m\ge m_0$, the vectors $\{\bm{1}_{B_{\alpha_j}\cap m}: j \in k\}$ are linearly independent in $\mathbb F_2^m$. Fix now $m\ge m_0$ and consider the linear map $T_m:\mathbb F_2^m\longrightarrow \mathbb F_2^k$ defined by
$$
\forall \tau \in \mathbb F_2^m, \quad 
T_m(\tau)
:=
\left(
\sum_{i\in B_{\alpha_j}\cap m}\tau(i): j \in k
\right).
$$ 
By the linear independence of $\{\bm{1}_{B_{\alpha_j}\cap m}: j \in k\}$, $T_m$ has rank $k$, hence $T_m$ is surjective and by rank-nullity theorem $|\mathrm{Ker}(T_m)|=2^{m-k}$. 
Hence, for every $\sigma\in 2^k$, exactly $2^{m-k}$ many $\tau\in 2^m$ satisfy $\Bigl(\sum_{i\in B_{\alpha_j}\cap m}\tau(i): j \in k\Bigr)=\sigma$. 
Therefore, on each block $I_m$ with $m\ge m_0$, every atom
$\bigcap_{j\in k}  A_{\alpha_j}^{\sigma(j)}$ occupies exactly a $2^{-k}$-proportion of $I_m$. 
Since $2^m=o\left(\sum_{j \in m}|I_j|\right)$ and $\sum_{j \in m}|I_j|=o(|I_m|)$ as $m\to \infty$, it easily follows that
$$
\forall \sigma \in 2^k, \quad 
\mathsf d\!\left(\bigcap_{j\in k} A_{\alpha_j}^{\sigma(j)}\right)=2^{-k}.
$$
In particular, each $A_\alpha$ has density $\nicefrac{1}{2}$ and $\mathcal{A}$ is a $\mathsf{d}$-independent family with cardinality $\mathfrak{c}$. The latter conclusion will be strengthened in several directions, see Proposition \ref{prop:continuum_indep}, Theorem \ref{thm:jon}, Corollary \ref{cor:existenceonehalfmaximal}, and Theorem \ref{thm:imagefieldsfidtsinct} below. 
\end{rmk}

Our main results are given in Section \ref{sec:mainresults}. The proofs are presented in Section \ref{sec:proofs}. 

%%%%%%%%%%%%%%%%%%%%%%%%%%%%%%%%%%%%%%%%%%%
\section{Main results}\label{sec:mainresults}

Our first result proves that, for each nonempty $S\subseteq (0,1)$, there exists a $\mathsf{d}$-independent family $\mathcal{A}$ such that $|\mathcal{A}|=\mathfrak{c}$ and $\mathsf{d}[\mathcal{A}]=S$.

\begin{prop}\label{prop:continuum_indep}
Fix $(p_\alpha: \alpha<\mathfrak{c})$ with values in $(0,1)$. Then there exists a $\mathsf d$-independent family $\{A_\alpha: \alpha<\mathfrak{c}\}$ such that $\mathsf{d}(A_\alpha)=p_\alpha$ for each $\alpha<\mathfrak{c}$.
\end{prop}
As a special case, we recover a recent result by Frankiewicz and Jureczko where they show in \cite[Lemma 9]{Frank21} that there exists a $\mathsf{d}$-independent family $\mathcal{A}$ of cardinality $\mathfrak{c}$ such that $\mathsf{d}(A)=\nicefrac{1}{2}$ for all $A \in \mathcal{A}$ (using different techniques); a related weaker claim appeared also in \cite[Corollary 4.11]{MR4609469}. %We refer also to \cite{MR1388771} for a connection between equidistributed sequences and asymptotic density, which is used in the proof of Proposition \ref{prop:continuum_indep}.

Observe that, if $(\mathcal{A}_\alpha: \alpha<\kappa)$ is an increasing chain of $\mathsf{d}$-independent families, then $\bigcup_\alpha \mathcal{A}_\alpha$ is still $\mathsf{d}$-independent. Hence, by Zorn's lemma, 
maximal $\mathsf{d}$-independent families exist. Accordingly, let us define 
$$
\mathscr{M}:=\left\{\mathcal{A}\subseteq \mathcal{D}: \mathcal{A} \text{ is a maximal }\mathsf{d}\text{-independent family}\right\}
$$
and denote by $\mathfrak{i}_{\mathsf{d}}$ the least cardinality of a maximal $\mathsf{d}$-independent family. 
\begin{rmk}\label{rmk:infinite}
Note that each $\mathcal{A} \in \mathscr{M}$ is infinite. Indeed, if $\mathcal{A}:=\{A_i: i \in k\}$ is a $\mathsf{d}$-independent family then for each $\sigma \in 2^k$ there exists a subset $B_\sigma$ of $A_\sigma:=\bigcap_i A_i^{\sigma(i)}\in \mathcal{D}$ for which $B_\sigma\in \mathcal{D}$ and $\mathsf{d}(B_\sigma)=\frac{1}{2}\mathsf{d}(A_\sigma)$.\footnote{In fact, if $X\in \mathcal{D}$ is an infinite set with increasing enumeration $(x_n: n \in \omega)$ then $\mathsf{d}(\{x_{2n}: n \in \omega\})=\mathsf{d}(X)/2$. This property has been termed \textquotedblleft thinnability\textquotedblright\, by van Douwen in \cite[p.225]{MR1192311}; cf. also \cite{MR3594409, 
%}. %
%, 
MR3879311}.
} 
It follows that $B:=\bigcup_\sigma B_\sigma \in \mathcal{D}$, $\mathsf{d}(B)=\nicefrac{1}{2}$, and the enlarged family $\mathcal{A}\cup \{B\}$ is also $\mathsf{d}$-independent. This claim will be improved in 
Corollary \ref{cor:ordercardinals} and 
Corollary \ref{cor:extension1} 
below, 
where we obtain that every maximal $\mathsf{d}$-independent family is necessarily uncountable and that every $\mathsf{d}$-independent family with cardinality smaller than $\mathfrak{i}_{\mathsf{d}}$ can be enlarged with another set in $\mathcal{D}$ with prescribed value in $(0,1)$. 
\end{rmk}

    Following \cite[Definition 2.2 and Definition 3.4]{MR4609469}, we write $\mathfrak{r}_{\nicefrac{1}{2}}$ for the least cardinality of a $\nicefrac{1}{2}$-reaping family $\mathcal{R}\subseteq [\omega]^\omega$, that is, a family for which there is no infinite $S \subseteq \omega$ such that 
    \begin{equation}\label{eq:bisect-R}
\lim_{n\to\infty}\frac{|S\cap B\cap n|}{|B\cap n|}=\frac12
\qquad\text{for all } B\in\mathcal{R}.
\end{equation} 

Recall from \cite[Theorem 3.8]{MR4609469} that $\aleph_1 \le \mathfrak{r}_{\nicefrac{1}{2}}$ and that, consistently, $\aleph_1 < \mathfrak{r}_{\nicefrac{1}{2}}$. 
\begin{thm}\label{thm:1}
    $\mathfrak{r}_{\nicefrac{1}{2}}\le \mathfrak{i}_{\mathsf{d}}$.
\end{thm}

Following the notation in \cite{MR4609469}, let $\mathcal{N}$ be the family of subsets $S\subseteq 2^\omega$ such that $\lambda(S)=0$, where $\lambda$ stands for the completion of the usual product measure on $2^\omega$. Let also $\mathrm{cov}(\mathcal{N})$ be the least cardinality of a family $\mathcal{S}\subseteq \mathcal{N}$ such that $\bigcup \mathcal{S}=2^\omega$. 
Results which are related to covering numbers and the ideal of asymptotic density zero sets can be found e.g. in \cite{MR4099835, MR3615051}. 
\begin{cor}\label{cor:ordercardinals}
    $\aleph_1 \le \mathrm{cov}(\mathcal{N}) \le \mathfrak{r}_{\nicefrac{1}{2}}\le \mathfrak{i}_{\mathsf{d}} 
    \le \mathfrak{c}$. 
    
    In addition, we have consistently $\aleph_1 < \mathrm{cov}(\mathcal{N})$ 
    and $\mathfrak{r}_{\nicefrac{1}{2}}<\mathfrak{c}$. 
\end{cor}

We remark that it is an open question whether it is consistent that $\mathrm{cov}(\mathcal{N})<\mathfrak{r}_{\nicefrac{1}{2}}$, see \cite[Question C]{MR4959539}. 

In the opposite direction, thanks to  Proposition \ref{prop:continuum_indep} there exists a maximal $\mathsf{d}$-independent family $\mathcal{A}$ of cardinality $\mathfrak{c}$ (more precisely, 
containing a subfamily of cardinality $\mathfrak{c}$ with a prescribed image in $(0,1)$). In particular: 
\begin{cor}\label{cor:maximalimage}
    There exists $\mathcal{A} \in \mathscr{M}$ with $\mathsf{d}[\mathcal{A}]=(0,1)$. 
\end{cor}
One might conjecture that the conclusion of Corollary \ref{cor:maximalimage} holds for every maximal $\mathsf{d}$-independent family $\mathcal{A}$. 
In the following result, we show that the answer is negative by proving that  \textquotedblleft most\textquotedblright\, of the image $\mathsf{d}[\mathcal{A}]$ can be prescribed to a given nonempty subset of $(0,1)$. Hereafter, given a subset $A\subseteq \mathbb{R}$, we denote by $A^\prime$ the set of its accumulation points. 
\begin{thm}\label{thm:jon}
     Pick a nonempty $S\subseteq (0,1)$. Then there exists 
    a maximal $\mathsf{d}$-independent family $\mathcal{A} \in \mathscr{M}$ with cardinality $\mathfrak{c}$ and an at most countable subfamily $\mathcal{B}\subseteq \mathcal{A}$ such that 
    $$
    \mathsf{d}[\mathcal{A}\setminus \mathcal{B}]=S 
    \qquad \text{ and }\qquad 
    \mathsf{d}[\mathcal{B}]^\prime \subseteq \{0,1\}.
    $$
\end{thm}

At this point, one might ask whether it is possible to choose $\mathcal{B}=\emptyset$. 
To this aim, we will need the following extension result. 
Let us recall $\mathrm{cov}(\mathcal{N})=\mathfrak{c}$ is implied by Martin's axiom $\mathsf{MA}$, or also by the Continuum Hypothesis $\mathsf{CH}$. 
\begin{prop}\label{prop:notextension2}
    Pick $\mathcal{A}\subseteq \mathcal{D}$ and $A \in \mathcal{D}\setminus \mathcal{A}$ such that $\mathcal{A}\cup \{A\}$ is $\mathsf{d}$-independent. Fix also $s \in (0,1)$. Then the following hold\textup{.}
    \begin{enumerate}[label={\rm (\roman*)}]
    \item \label{item1extension} If $|\mathcal{A}|\le \omega$ then there exists $B \in \mathcal{D}$ such that $\mathsf{d}(B)=s$, $\mathcal{A}\cup \{B\}$ is $\mathsf{d}$-independent, and $\{A,B\}$ is not $\mathsf{d}$-independent\textup{;}
    \item \label{item2extension} Assume $\mathrm{cov}(\mathcal{N})=\mathfrak{c}$. If $|\mathcal{A}|<\mathfrak{c}$  there exists $B \in \mathcal{D}$ such that $\mathsf{d}(B)=s$, $\mathcal{A}\cup \{B\}$ is $\mathsf{d}$-independent, and $\{A,B\}$ is not $\mathsf{d}$-independent\textup{.}
    \end{enumerate}  
\end{prop}
As we highlight in Remark \ref{rmk:perfectset} below, 
%denoting with $\mathcal{Z}$ the family of asymptotic density
it turns out that 
it is possible to choose a perfect subset $\mathcal{P}\subseteq 2^\omega$ of sets $B \in \mathcal{D}$ satisfying the claim of Proposition \ref{prop:notextension2}\ref{item1extension} such that, in addition, the sets in $\mathcal{P}$ are pairwise distinct modulo asymptotic density zero modifications. 

Taking into account that the least size of a maximal $\mathsf{d}$-independent family $\mathcal{A} \in \mathscr{M}$ is $\ge \mathrm{cov}(\mathcal{N})$ by Corollary \ref{cor:ordercardinals}, we obtain: 
\begin{cor}\label{cor:extension1}
    Let $\mathcal{A}$ be a $\mathsf{d}$-independent family and pick $s \in (0,1)$. Then the following hold\textup{.}
    \begin{enumerate}[label={\rm (\roman*)}]
    \item \label{item1extensioncc} If $|\mathcal{A}|\le \omega$ then there exists $B \in \mathcal{D}$ such that $\mathsf{d}(B)=s$ and $\mathcal{A}\cup \{B\}$ is $\mathsf{d}$-independent\textup{;}
    \item \label{item2extensioncc} Assume $\mathrm{cov}(\mathcal{N})=\mathfrak{c}$. If $|\mathcal{A}|<\mathfrak{c}$  there exists $B \in \mathcal{D}$ such that $\mathsf{d}(B)=s$ and $\mathcal{A}\cup \{B\}$ is $\mathsf{d}$-independent\textup{.}
    \end{enumerate}  
\end{cor}

Using recursively the above results, 
%our last result proves 
we show 
the existence of a maximal $\mathsf{d}$-independent family with a prescribed set of values in $(0,1)$. 
\begin{thm}\label{thm:mainprescribedimage}
    Assume $\mathrm{cov}(\mathcal{N})=\mathfrak{c}$. 
    Pick a nonempty $S\subseteq (0,1)$. Then there exists 
    a maximal $\mathsf{d}$-independent family $\mathcal{A} \in \mathscr{M}$ such that $\mathsf{d}[\mathcal{A}]=S$.
\end{thm}

The following consequence is immediate: 
\begin{cor}\label{cor:existenceonehalfmaximal}
    Assume $\mathrm{cov}(\mathcal{N})=\mathfrak{c}$. 
    Then there exists a maximal $\mathsf{d}$-independent family $\mathcal{A} \in \mathscr{M}$ such that $\mathsf{d}(A)=\nicefrac{1}{2}$ for all $A \in \mathcal{A}$. 
\end{cor}

Now, observe that if $\mathcal{A}$ is a maximal $\mathsf{d}$-independent family then $\mathcal{A}_{\mathcal{S}}:=\{A^0: A \in \mathcal{S}\}\cup \{A^1: A \in \mathcal{A}\setminus \mathcal{S}\}$ is maximal $\mathsf{d}$-independent as well for each $\mathcal{S}\subseteq \mathcal{A}$. However, $\langle \mathcal{A}\rangle=\langle \mathcal{A}_{\mathcal{S}}\rangle$, and hence $\mathsf{d}[\langle \mathcal{A}\rangle]=\mathsf{d}[\langle \mathcal{A}_{\mathcal{S}}\rangle]$, for each $\mathcal{S}\subseteq \mathcal{A}$. 
In the next result, we show that $\mathsf{d}[\langle \mathcal{A}\rangle]$ may attain many subsets of $[0,1]$. % as $\mathcal{A}\in \mathscr{M}$. 
\begin{thm}\label{thm:imagefieldsfidtsinct}
There exist $2^{\mathfrak{c}}$ maximal $\mathsf d$-independent families 
$ 
    \{\mathcal{A}_\alpha:\alpha <2^{\mathfrak c}\}\subseteq \mathscr M
$ 
such that $|\mathcal{A}_\alpha|=\mathfrak{c}$ for each $\alpha<2^{\mathfrak{c}}$ and the sets  
$$
\{\mathsf{d}[\langle \mathcal{A}_\alpha\rangle]: \alpha<2^\mathfrak{c}\}
$$
are pairwise distinct.   
\end{thm}

As it follows by Remark \ref{rmk:dense} below, it is worth noting that $\mathsf{d}[\langle \mathcal{A}\rangle]$ is dense in $[0,1]$ for all maximal $\mathsf{d}$-independent families $\mathcal{A}$. 
%As we observe in Remark \ref{rmk:dense} below, if $\mathcal{A}$ is an uncountable $\mathsf{d}$-independent family (which is the case if $\mathcal{A}$ is maximal by  Corollary \ref{cor:ordercardinals}), then $\mathsf{d}[\langle \mathcal{A}\rangle]$ is dense in $[0,1]$. 
%A remark is in order about the largeness of $\mathsf{d}[\langle \mathcal{A}\rangle]$:

As a consequence of Theorem \ref{thm:imagefieldsfidtsinct}, we obtain that there are $2^{\mathfrak{c}}$ maximal $\mathsf{d}$-independent families which generate pairwise distinct fields: 
\begin{cor}\label{cor:quantity2c}
$
|\{\mathsf{d}[\langle \mathcal{A}\rangle]: \mathcal{A} \in \mathscr{M}\}|=
|\{\langle \mathcal{A}\rangle: \mathcal{A} \in \mathscr{M}\}|=|\mathscr{M}|=2^\mathfrak{c}$.
\end{cor}

Let us identify each infinite set $A\subseteq \omega$ with the real number $\sum_{n \in A}{2^{-n-1}} \in (0,1]$. Hence each family $\mathcal{F}$ of infinite subsets of $\omega$ can be regarded as a subset of $(0,1]$. In particular, we can speak about its topological complexity. 
\begin{cor}\label{cor:notanalytic}
    There exists a maximal $\mathsf{d}$-independent family $\mathcal{A}\in \mathscr{M}$ with cardinality $\mathfrak{c}$ such that neither $\mathcal{A}$ nor $\langle \mathcal{A}\rangle$ nor $\mathsf{d}[\langle \mathcal{A}\rangle]$ is analytic \textup{(}hence, not Borel\textup{)}. 
\end{cor}

%Since there are $2^{\mathfrak{c}}$ many subsets of $2^\omega$ with the Baire property, it is evident that the same strategy of the proof of Corollary \ref{cor:notanalytic} would fail to prove the existence of some maximal $\mathsf{d}$-independent family without the Baire property. 
Lastly, using different methods, the existence of a nonanalytic family $\mathcal{A}\in \mathscr{M}$ can be strengthened in two directions.
\begin{thm}\label{thm:notbaire}
    There exists a maximal $\mathsf{d}$-independent family 
    $\mathcal{A}\in \mathscr{M}$ 
    with cardinality $\mathfrak{c}$ which does not have the Baire property.  
\end{thm}
\begin{thm}\label{thm:notmeasurable}
    Assume $\mathrm{cov}(\mathcal{N})=\mathfrak{c}$.  
    Then there exists a maximal $\mathsf{d}$-independent family 
    $\mathcal{A}\in \mathscr{M}$ 
    with cardinality $\mathfrak{c}$ which is not Lebesgue measurable. 
\end{thm}

We leave 
as an open question whether 
\emph{every} maximal $\mathsf{d}$-independent family does not have the Baire property (or simply, is not Borel) and/or is not Lebesgue measurable. 
Moreover, taking into account Corollary \ref{cor:existenceonehalfmaximal}, we also leave as an open question whether 
it is consistent in $\mathsf{ZFC}$ that there are no maximal $\mathsf{d}$-independent families $\mathcal{A} \in \mathscr{M}$ such that $\mathsf{d}(A)=\nicefrac{1}{2}$ for all $A \in \mathcal{A}$.

%\textcolor{red}{Possibly add: is it true that there exists a maximal $\mathcal{A}$ such that neither $\mathcal{A}$ nor $\langle \mathcal{A}\rangle$ has the Baire property? (Note that there are $2^{\mathfrak{c}}$ subsets of $[0,1]$ with the Baire property.)}

%%%%%%%%%%%%%%%%%%%%%%%%%%%%%%%%%%%%%%%%
\section{Proofs}\label{sec:proofs}

\begin{proof}
[Proof of Proposition \ref{prop:continuum_indep}] 
Let $H\subseteq\mathbb R$ be a Hamel basis of $\mathbb R$ over $\mathbb Q$ such that $1\in H$, and define $I:=H\setminus\{1\}$. 
Pick an enumeration $\{i_\alpha: \alpha < \mathfrak{c}\}$ of $I$, and observe that $\{1\}\cup F$ is $\mathbb Q$-linearly independent for every finite $F\subseteq I$. Now, for each $\alpha<\mathfrak{c}$ define
\[
A_\alpha:=
\left\{n\in\omega: \{ni_\alpha\}\in[0,p_\alpha)\right\},
\]
where $\{x\}$ denotes the fractional part of $x$. By Kronecker--Weyl's equidistribution theorem, $A_\alpha \in \mathcal{D}$ and $\mathsf{d}(A_\alpha)=p_\alpha$ for each $\alpha <\mathfrak{c}$. To complete the proof, it is enough to show that the family  $\mathcal A:=\{A_\alpha:\alpha<\mathfrak{c}\}$ is $\mathsf d$-independent.

To this aim, pick a nonempty finite $F\subseteq \mathfrak{c}$ and define $k:=|F|$. 
Since $\{1\}\cup \{i_\alpha: i_\alpha \in F\}$ is $\mathbb Q$-linearly independent, the sequence 
$$
\left(\big(\{ni_\alpha\}: \alpha \in F\big): n \in \omega\right)
$$
is uniformly distributed in $[0,1)^k$, thanks again to Kronecker--Weyl's theorem. Hence, for every $\sigma\in 2^F$, the set $A_\sigma:=\bigcap_{\alpha\in F} A_\alpha^{\sigma(\alpha)}$ admits asymptotic density, which is equal to the $k$-dimensional Lebesgue measure of the corresponding rectangle $R_\sigma:=\prod_{\alpha\in F}R^{\sigma(\alpha)}_\alpha$, where 
$R^1_\alpha:=[0,p_\alpha)$ and $R^0_\alpha:=[p_\alpha,1)$. In particular,
$$
\mathsf d\Big(\bigcap_{\alpha\in F}A_\alpha\Big)
=\prod_{\alpha \in F}p_\alpha 
=\prod_{\alpha\in F}\mathsf d(A_\alpha).
$$
Lastly, we only need to show that $\langle \mathcal{A}\rangle \subseteq \mathcal{D}$. In fact, if $A \in \langle \mathcal{A}\rangle$ then there exist a nonempty $F \in [I]^{<\omega}$ and a subset $E\subseteq 2^F$ such that $A$ can be partitioned into $\{A_\sigma: \sigma \in E\}$. The conclusion follows since $\mathcal{D}$ is stable under disjoint unions and $A_\sigma \in \mathcal{D}$ for each $\sigma \in E$. 
\end{proof}

\medskip

\begin{proof}
[Proof of Theorem \ref{thm:1}]
Fix $\mathcal{A}\in\mathscr{M}$ and define 
\[
\mathcal{R}:=\Big\{\bigcap \mathcal{F}: \emptyset \neq \mathcal{F}\in[\mathcal{A}]^{<\omega}\Big\}\cup\{\omega\}.
\] 
We claim that $\mathcal{R}$ is a $\nicefrac12$-reaping family. Since $\mathcal{A}$ is necessarily infinite, this will yield
$\mathfrak{r}_{\nicefrac{1}{2}}\le |\mathcal{R}|\le |\mathcal{A}|$, which would be enough to complete the proof. 

Suppose for the sake of contradiction that $\mathcal{R}$ is not $\nicefrac{1}{2}$-reaping. Then there exists an infinite
$S\subseteq\omega$ which satisfies \eqref{eq:bisect-R}. 
In particular, choosing $B=\omega$, we get $S\in\mathcal{D}$ and $\mathsf{d}(S)=\nicefrac{1}{2}$. 
Let now $\mathcal{F}\subseteq \mathcal{A}$ be a nonempty finite subfamily and define $F:=\bigcap\mathcal{F} \in \mathcal{D}$. Thanks to 
\eqref{eq:bisect-R} and the previous observation, we obtain 
$$
\mathsf d(S\cap F)
=\lim_{n\to\infty}\frac{|S\cap F\cap n|}{|F\cap n|}
\cdot \lim_{n\to\infty}\frac{|F\cap n|}{n}
=\frac12\,\mathsf d(F)=\mathsf{d}(S) \prod_{A \in \mathcal{F}}\mathsf{d}(A).
$$

It easily follows that $\mathcal{A}\cup\{S\}$ is a $\mathsf d$-independent family. This contradicts the maximality of $\mathcal{A}$. Therefore $\mathfrak{r}_{\nicefrac{1}{2}}\le \mathfrak{i}_{\mathsf{d}}$. 
\end{proof}

\medskip

\begin{proof}
[Proof of Corollary \ref{cor:ordercardinals}]
It 
follows by 
Theorem \ref{thm:1} and \cite[Theorem 3.8]{MR4609469}. 
\end{proof}

\medskip

\begin{proof}
[Proof of Theorem \ref{thm:jon}]
Pick $(p_\alpha:\alpha<\mathfrak{c})$ with values in $S$ and such that
$\{p_\alpha:\alpha<\mathfrak{c}\}=S$. By Proposition \ref{prop:continuum_indep} there exists a $\mathsf d$-independent family
$\mathcal{A}^\ast=\{A_\alpha:\alpha<\mathfrak{c}\}$ such that $\mathsf d(A_\alpha)=p_\alpha$ for all $\alpha<\mathfrak{c}$.
In particular, $|\mathcal{A}^\ast|=\mathfrak{c}$ and $\mathsf d[\mathcal{A}^\ast]=S$. 
Define 
$$
\mathscr{P}:=\Big\{\mathcal{F}\subseteq \mathcal{D}:\ \mathcal{F}\text{ is }\mathsf d\text{-independent, } \mathcal{A}^\ast\subseteq \mathcal{F},
\text{ and }\mathsf d[\mathcal{F}]=S\Big\},
$$
partially ordered by inclusion. Since $\mathscr{P}$ is stable under arbitrary union of increasing chains, it follows by Zorn's lemma that there exists a maximal element, let us say $\mathcal{A}_0\in\mathscr{P}$.
Notice that $\mathcal{A}^\ast\subseteq\mathcal{A}_0$, hence  $|\mathcal{A}_0|=\mathfrak{c}$, and $\mathsf d[\mathcal{A}_0]=S$.

At this point, let $\mathcal{A}\in\mathscr{M}$ be a maximal $\mathsf d$-independent family with $\mathcal{A}_0\subseteq\mathcal{A}$, 
which exists again by Zorn's lemma, and define 
$$
\mathcal{B}:=\mathcal{A}\setminus\mathcal{A}_0.
$$
It follows by construction that $|\mathcal{A}|=\mathfrak{c}$ and $\mathsf d[\mathcal{A}\setminus\mathcal{B}]=\mathsf d[\mathcal{A}_0]=S$. Hence, it remains to show that $\mathcal{B}$ is at most countable and that $\mathsf d[\mathcal{B}]^\prime\subseteq\{0,1\}$.

To conclude the proof, it will be enough to show that, for each $n\ge 3$, the set 
$$
\mathcal{B}_n:=\Big\{B\in\mathcal{B}: \tfrac1n\le \mathsf d(B)\le 1-\tfrac1n\Big\}
$$
is finite.
Indeed, in such case, since $\mathsf{d}(B) \in (0,1)$ for all $B \in \mathcal{B}$ and $\mathcal{B}=\bigcup_{n\ge 3}\mathcal{B}_n$, then $\mathcal{B}$ is at most countable. In addition, if $r \in (0,1)$ were an accumulation point of  $\mathsf d[\mathcal{B}]$, then $r\in(\nicefrac{1}{n},1-\nicefrac{1}{n})$ for some $n\ge 3$, and therefore $\mathcal{B}_n$ would be infinite, which is a contradiction. Therefore $\mathsf d[\mathcal{B}]^\prime\subseteq\{0,1\}$. 

Hence, pick an integer $n\ge 3$. 
Suppose towards a contradiction that there exists an integer $n\ge 3$ such that $\mathcal{B}_n$ is infinite, and pick pairwise distinct $\{B_k: k\in\omega\}$ in $\mathcal{B}_n$. Fix also $s\in S$ and define $q_k:=\mathsf{d}(B_k)\in (0,1)$ for all $k \in \omega$ and 
$$
\varepsilon:=\tfrac12\min\{q_0(1-s),s(1-q_0)\}>0.
$$
Observe that 
$$
x_0:=s(1-q_0)-\varepsilon \in (0,s) 
\quad \text{ and }\quad 
x_1:=sq_0+\varepsilon\in (0,q_0)
.
$$
Define also $B_{m,\sigma}:=\bigcap_{i<m} B_i^{\sigma(i)}$ for all $m \in \omega$ and $\sigma \in 2^m$. 
Since $\langle \mathcal{B}\rangle \subseteq \langle \mathcal{A}\rangle\subseteq \mathcal{D}$, then each $B_{m,\sigma}$ belongs to $\mathcal{D}$. In addition, for all $m \in \omega$ and $\sigma \in 2^m$ it follows by \eqref{eq:d3sharp} that
$$
\mathsf{d}(B_{m,\sigma})=
\prod_{i<m}\big(\sigma(i)q_i+(1-\sigma(i))(1-q_i)\big)
\le \left(1-\frac{1}{n}\right)^m.
$$
Now, let $(m_\ell: \ell \in \omega)$ be a strictly increasing sequence of positive integers such that $(1-\nicefrac{1}{n})^{m_\ell}<2^{-\ell}$ for all $\ell \in \omega$, 
so that by construction $\mathsf{d}(B_{m_\ell,\sigma})<2^{-\ell}$ for all $\ell \in \omega$ and $\sigma \in 2^{m_\ell}$. 
\begin{claim}\label{claim:subsetB0i}
There exist $X_0\subseteq B_0^0$ and $X_1\subseteq B_0^1$ in $\mathcal{D}$ such that
$$
\mathsf{d}(X_i)=x_i
\quad \text{ and }\quad 
\mathsf d\Big(X_i\cap\bigcap\mathcal{F}\Big)=\mathsf d(X_i)\,\mathsf d\Big(\bigcap\mathcal{F}\Big)
$$
for every finite nonempty 
$\mathcal{F}\subseteq \mathcal{A}_0$ and $i \in 2$. 
\end{claim}
\begin{proof}
For each $X,Y\subseteq \omega$ define $\rho(X,Y):=\mathsf{d}^\star(X\bigtriangleup Y)$, where $\mathsf{d}^\star$ stands for the upper asymptotic density on $\omega$. It follows by \cite[Theorem 1.1]{MR4868961} that $(\mathcal{P}(\omega), \rho)$ is a complete pseudometric space, and that $\mathcal{D}$ is $\rho$-closed; cf. also \cite[Theorem 2.4 and Example 3.3]{KeithLeo24}. 
Now, fix $i \in 2$, set for convenience $\Sigma_{i,-1}:=\emptyset$ and $m_{-1}:=0$, and for each $\ell \in \omega$ pick recursively a subset $\Sigma_{i,\ell}\subseteq \{\sigma\in  2^{m_\ell}: \sigma(0)=i\}$ with maximal cardinality such that
$$
\left\{\sigma \in 2^{m_\ell}: \sigma \upharpoonright m_{\ell-1} \in \Sigma_{i,\ell-1}\right\}\subseteq \Sigma_{i,\ell} 
\quad \text{ and }\quad 
\mathsf{d}(X_{i,\ell})<x_i,
$$
where $X_{i,\ell}:=\bigcup \left\{B_{m_\ell,\sigma}: \sigma \in \Sigma_{i,\ell} \right\}\in \mathcal{D}$. 
It follows by construction that the sequence $(X_{i,\ell}: \ell \in \omega)$ is increasing in $\mathcal{D}$ and that, by maximality hypothesis, it is $\rho$-Cauchy. Hence it is $\rho$-convergent to some $X_i\subseteq \omega$. 
Since the family $\mathcal{D}$ is $\rho$-closed in $\mathcal{P}(\omega)$, then $X_i \in \mathcal{D}$ and $\mathsf{d}(X_i)=x_i$. 

To conclude the proof of the claim, pick a finite nonempty $\mathcal{F}\subseteq \mathcal{A}_0$ and set $D:=\bigcap\mathcal{F}\in\langle\mathcal{A}_0\rangle\subseteq\mathcal{D}$. 
Fix $i\in 2$. 
Since $\langle\mathcal{A}\rangle\subseteq\mathcal{D}$, we get by the definition of $X_{i,\ell}$ and the finite additivity of $\mathsf d$ on $\mathcal{D}$ that
\[
\forall \ell \in \omega, \quad \mathsf d(X_{i,\ell}\cap D)
=\sum_{\sigma\in\Sigma_{i,\ell}} \mathsf d(B_{m_\ell,\sigma}\cap D).
\]
Now, recall that $\{\mathcal{A}_0, \mathcal{B}\}$ is a partition of $\mathcal{A}$, and that $\mathcal{A}$ is $\mathsf{d}$-independent. Since each $B_{m_\ell,\sigma} \in \langle \mathcal{B}\rangle$ and $\mathcal{F}\subseteq \mathcal{A}_0$, we obtain that 
\begin{equation}\label{eq:limitsssss}
\forall \ell \in \omega, \quad \mathsf d(X_{i,\ell}\cap D)
=\sum_{\sigma\in\Sigma_{i,\ell}} \mathsf d(B_{m_\ell,\sigma})\mathsf{d}(D)
=\mathsf d(X_{i,\ell})\mathsf{d}(D).
\end{equation}
Finally, observe that $(X_{i,\ell}\cap D)\triangle (X_i\cap D)\subseteq X_{i,\ell}\triangle X_i$ implies $\rho(X_{i,\ell}\cap D, X_i\cap D)\le \rho(X_{i,\ell},X_i)$. Since $(X_{i,\ell}: \ell \in \omega)$ is $\rho$-convergent to $X_i$ then 
$\lim_\ell \mathsf{d}(X_{i,\ell})=\mathsf{d}(X_i)$ 
and 
$\lim_{\ell}\mathsf d(X_{i,\ell}\cap D)=\mathsf d(X_{i}\cap D)$. 
Taking the limit as $\ell\to \infty$ in \eqref{eq:limitsssss}, we obtain 
$$
\mathsf d(X_{i}\cap D)
=\lim_{\ell\to\infty}\mathsf d(X_{i,\ell}\cap D)
=\lim_{\ell\to\infty}\mathsf d(X_{i,\ell})\mathsf{d}(D)
=\mathsf d(X_{i})\mathsf{d}(D).
$$  
This completes the proof of the claim.
\end{proof}

To conclude the proof, define $X:=X_0\cup X_1$. Since $\mathcal{D}$ is stable under disjoint unions and $\mathsf{d}$ is finitely additive on $\mathcal{D}$, then $X\in\mathcal{D}$ and
\[
\mathsf d(X)=\mathsf d(X_0)+\mathsf d(X_1)=x_0+x_1=s\in S.
\]
Moreover, for every nonempty finite $\mathcal{F}\subseteq\mathcal{A}_0$, we get by Claim \ref{claim:subsetB0i} 
$$
\mathsf d\Big(X\cap\bigcap\mathcal{F}\Big)
=\mathsf d\Big(X_0\cap\bigcap\mathcal{F}\Big)+\mathsf d\Big(X_1\cap\bigcap\mathcal{F}\Big)
=\mathsf d(X)\,\mathsf d\Big(\bigcap\mathcal{F}\Big).
$$
Hence the family $\mathcal{A}_0\cup\{X\}$ is $\mathsf d$-independent, it contains $\mathcal{A}^\star$, and $\mathsf{d}[\mathcal{A}_0 \cup \{X\}]=S$. 
Since $X\cap B_0=X\cap B_0^1=X_1 \in \mathcal{D}$, it follows by Claim \ref{claim:subsetB0i} that
$$
\mathsf d(X\cap B_0)=x_1=sq_0+\varepsilon \neq sq_0=\mathsf d(X)\mathsf d(B_0),
$$
so that $\{X,B_0\}$ is not $\mathsf{d}$-independent. 
However, the $\mathsf{d}$-independence of $\mathcal{A}$ implies that $\{Y,B_0\}$ is $\mathsf{d}$-independent for all $Y \in \mathcal{A}_0$. Hence $X\notin \mathcal{A}_0$ and $\mathcal{A}_0\cup\{X\}$ is a proper extension of $\mathcal{A}_0$ in $\mathscr{P}$. 
This contradicts the maximality of $\mathcal{A}_0$. 
Therefore $\mathcal{B}_n$ must be finite.
\end{proof}

\medskip

\begin{proof}
[Proof of Proposition \ref{prop:notextension2}]
We start proving item \ref{item2extension} first.

\medskip

\ref{item2extension}.
Suppose $\mathcal{A}$ is infinite (the case of $\mathcal{A}$ being finite is contained in item \ref{item1extension} below). Fix an enumeration $\mathcal A=\{A_i:i<\kappa\}$ with $\kappa<\mathfrak c$.
For each finite $F\subseteq \kappa$ and $\sigma\in 2^F$, define $A_{F,\sigma}:=\bigcap_{i\in F}A_i^{\sigma(i)}$ and set 
$$
\mathscr A:=\{A_{F,\sigma}:F\in[\kappa]^{<\omega},\ \sigma\in 2^F\}.
$$
Then $\omega\in\mathscr A$ (by choosing $F=\emptyset$), $|\mathscr A|=\kappa$, and
$\mathscr A\subseteq\langle\mathcal A\rangle\subseteq\mathcal D$.
Moreover, by \eqref{eq:d3sharp} we have $\mathsf d(D)\in(0,1)$ for every $D\in\mathscr A\setminus \{\omega\}$, and every
$S\in\langle\mathcal A\rangle$ is a finite disjoint union of members of $\mathscr A$.

Define $a:=\mathsf d(A)$, which belongs to $(0,1)$ since $\mathcal A\cup\{A\}$ is $\mathsf d$-independent.
Let $\varepsilon:=\frac{1}{2} \min\{a(1-s), s(1-a)\}$ and define
$$
x_1:=sa+\varepsilon,
\quad x_0:=s(1-a)-\varepsilon,
\quad t_1:=\frac{x_1}{a},
\,\,\text{ and }\,\, 
t_0:=\frac{x_0}{1-a}.
$$
Then $x_0,x_1,t_0,t_1\in(0,1)$, and $t_1a+t_0(1-a)=s$.

Identifying binary sequences in $2^{\omega}$ with the corresponding subsets of $\omega$, 
let $\nu$ be the product probability measure on $2^\omega$ defined by declaring the coordinates independent and
$$
\forall n \in \omega, \quad 
\nu(\{S\subseteq\omega:\ n\in S\})=
\begin{cases}
\,t_1,& \text{ if } n\in A,\\
\,t_0,& \text{ if } n\notin A.
\end{cases}
$$
For $S\subseteq\omega$ and $n\in\omega$, write $X_n(S):=\mathbf 1_S(n)\in\{0,1\}$ and set $p_n:=\mathbb E_\nu [X_n]$, so that $p_n=t_1$ if $n \in A$ (i.e., $n \in A^1$) and $p_n=t_0$ if $n\notin A$ (i.e., $n \in A^0$). Now, fix $D\in\mathscr A$ and define
$$
\forall n \in \omega, \quad 
Y_n:=\left(X_n-p_n\right)\mathbf 1_D(n).
$$
Since $D$ is fixed, each $Y_n$ is a measurable function of $X_n$, hence $(Y_n)_{n\in\omega}$ is a sequence of independent random variables.
Moreover, $\mathbb E_\nu[Y_n]=0$ and $\mathrm{Var}(Y_n)\le \mathbb E_\nu[(X_n-p_n)^2]\le \nicefrac{1}{4}$ for all $n\in \omega$.
Thus $\sum_{n\ge 1}\mathrm{Var}(Y_n)/n^2\ll \sum_{n\ge 1}1/n^2<\infty$, and by Kolmogorov's strong law of large numbers, see e.g. \cite[Corollary 5.22]{MR4226142}, we can fix $\mathcal G_D\subseteq\mathcal P(\omega)$ such that  $\nu(\mathcal G_D)=1$ and 
\begin{equation}\label{eq:SLLN}
\forall S\in\mathcal G_D,\qquad \lim_{n\to\infty}\frac1n\sum_{k\in n}Y_k(S)=0.
\end{equation}
Now, since $\mathcal A\cup\{A\}$ is $\mathsf d$-independent and $D\in\mathscr A\subseteq\langle\mathcal A\rangle$,
we have $D\cap A^1\in\mathcal D$ and $D\cap A^0 \in\mathcal D$, with 
\begin{equation}\label{eq:dDA}
\mathsf d(D\cap A^1)=a\,\mathsf d(D)
\quad \text{ and }\quad 
\mathsf d(D\cap A^0)=(1-a)\,\mathsf d(D).
\end{equation}
Therefore, using \eqref{eq:SLLN} and \eqref{eq:dDA}, we obtain that, for every $S\in\mathcal G_D$,
\begin{displaymath}
\begin{split}
\lim_{n\to\infty}\frac{|S\cap D\cap n|}{n}
&=\lim_{n\to\infty}\frac1n\sum_{k\in n}X_k(S)\mathbf 1_D(k)\\
&=\lim_{n\to\infty}\frac1n\sum_{k\in n}Y_k(S)
+\lim_{n\to\infty}\frac1n\sum_{k\in n}p_k\mathbf 1_D(k)\\
&=\lim_{n\to\infty}\left(
t_1\frac{|A^1\cap D\cap n|}{n}+t_0\frac{|A^0 \cap D\cap n|}{n}\right)\\
&=t_1\,\mathsf d(D\cap A^1)+t_0\,\mathsf d(D\cap A^0)\\
&=\big(t_1a+t_0(1-a)\big)\mathsf d(D)=s\,\mathsf d(D).
\end{split}
\end{displaymath}
In particular, we have
\begin{equation}\label{eq:SD-correct}
\forall S \in \mathcal G_D, \quad 
S\cap D\in\mathcal D
\quad\text{and}\quad
\mathsf d(S\cap D)=s\,\mathsf d(D).
\end{equation}
Applying this also with $D=\omega\in\mathscr A$, we obtain that $S\in\mathcal D$ and $\mathsf d(S)=s$ for every
$S\in\mathcal G_\omega$.
Finally, applying a similar argument with $D=A$, we obtain a set
$\mathcal G_A\subseteq 2^\omega$ with $\nu(\mathcal G_A)=1$ such that 
\begin{equation}\label{eq:SA}
\forall S \in \mathcal G_A, \quad 
S\cap A\in\mathcal D
\quad\text{and}\quad
\mathsf d(S\cap A)=t_1\,\mathsf d(A)=x_1.
\end{equation}

Define $\mathcal{N}_{\overline{\nu}}:=\{S\subseteq 2^\omega: \overline{\nu}(S)=0\}$, where $\overline{\nu}$ stands for the completion of $\nu$. It is straightforward to check that the probability space $(2^\omega, \overline{\nu})$ is atomless, countably separated, and compact. It follows by \cite[Corollary 344K]{MR2459668} that the measure space $(2^\omega, \overline{\nu})$ is isomorphic to $(2^\omega, \lambda)$. Hence their null ideals $\mathcal{N}_{\overline{\nu}}$ and $\mathcal{N}$ are isomorphic, which implies that their covering numbers coincide, i.e., $\mathrm{cov}(\mathcal{N})=\mathrm{cov}(\mathcal{N}_{\overline{\nu}})$.

At this point, since $|\mathscr{A}|=\kappa<\mathfrak{c}=\mathrm{cov}(\mathcal{N})=\mathrm{cov}(\mathcal{N}_{\overline{\nu}})$, it follows that 
\begin{equation}\label{eq:definitionGAGD}
    \mathcal G:=\mathcal G_A\ \cap\ \bigcap_{D\in\mathscr A}\mathcal G_D.
\end{equation}
is nonempty. 
Fix $B \in \mathcal{G}$.  
We now verify that $B$ has the desired properties.
First, by $B\in\mathcal G_\omega$ we have $B\in\mathcal D$ and $\mathsf d(B)=s$.
Moreover, by \eqref{eq:SA},
\[
\mathsf d(A\cap B)=x_1=sa+\varepsilon\neq sa=\mathsf d(B)\mathsf d(A),
\]
so $\{A,B\}$ is not $\mathsf d$-independent. 

To complete the proof, it remains to show that $\mathcal A\cup\{B\}$ is $\mathsf d$-independent.
Item~\ref{defi:1indep} holds for $B$ since $\mathsf d(B)=s\in(0,1)$.
For item~\ref{defi:2indep}, pick $S\in\langle\mathcal A\cup\{B\}\rangle$.
Then $S$ is a finite disjoint union of sets in $\mathscr{A}$ and sets of the form $B^\tau\cap D$ with $\tau\in 2$ and $D\in\mathscr A$.
By \eqref{eq:SD-correct} for any $D\in\mathscr A$ we have $B^1\cap D\in\mathcal D$ and $\mathsf d(B^1\cap D)=s\,\mathsf d(D)$; since also $D\in\mathcal D$,
the limit of
$$
\frac{|B^0 \cap D\cap n|}{n}
=\frac{|D\cap n|}{n}-\frac{|B^1\cap D\cap n|}{n}
$$
exists as $n\to \infty$, hence $B^0\cap D\in\mathcal D$.
Thus every set of the form $B^\tau\cap D$ belongs to $\mathcal D$, and since $\mathcal D$ is stable under finite disjoint unions,
it follows that $S\in\mathcal D$. Therefore $\langle\mathcal A\cup\{B\}\rangle\subseteq\mathcal D$.

Finally, for item~\ref{defi:3indep}, pick a nonempty finite $F\subseteq \kappa$ and pick $\sigma\in 2^F$ with $\sigma(i)=1$ for all $i \in F$. 
Then $D:=A_{F,\sigma}\in\mathscr A$, and by \eqref{eq:SD-correct} we have
\[
\mathsf d\Big(B\cap \bigcap_{i\in F}A_i\Big)=\mathsf d(B\cap D)=s\,\mathsf d(D)
%=\mathsf d(B)\,\mathsf d\Big(\bigcap_{i\in F}A_i\Big)
=\mathsf d(B)\prod_{i\in F}\mathsf d(A_i),
\]
where the last equality uses the $\mathsf d$-independence of $\mathcal A$.
This completes the proof of item \ref{item2extension}.

\medskip

\ref{item1extension}. 
This goes verbatim, except that, if $|\mathcal{A}|\le \omega$, then $|\mathscr{A}|\le \omega$, so that the set $\mathcal{G}$ defined in \eqref{eq:definitionGAGD} has full $\nu$-measure (without any additional hypothesis). 
\end{proof}

\medskip

\begin{rmk}\label{rmk:perfectset}
Let $\mathsf{d}^\star$ be the upper asymptotic density on $\omega$ and define $$
\mathcal{Z}:=\{S\subseteq \omega: \mathsf{d}^\star(S)=0\}.
$$
Then $\mathcal{Z}$ is an analytic ideal on $\omega$ (regarded as a subspace of the Cantor space $2^\omega$), see e.g. \cite[Example 1.2.3(d)]{MR1711328}, hence it is meager, see e.g. \cite[Theorem 2.4 and Proposition 2.5]{FKL2025} and references therein. 

We claim that, if $\mathcal{A}$, $A$, and $s$ are as in the statement of Proposition \ref{prop:notextension2}\ref{item1extension} (hence, with $|\mathcal{A}|\le \omega$) then there exists a nonempty perfect $\mathcal{P}\subseteq 2^\omega$ such that:
\begin{enumerate}[label={\rm (\roman*)}]
    \item \label{item:rmk1perfect} $\mathcal{P}\subseteq \mathcal{D};$
    \item \label{item:rmk2perfect} $\mathsf{d}(B)=s$, $\mathcal{A}\cup \{B\}$ is $\mathsf{d}$-independent, and $\{A,B\}$ is not $\mathsf{d}$-independent for all $B \in \mathcal{P}$;
    \item \label{item:rmk3perfect} $B\bigtriangleup C\notin \mathcal{Z}$ for all distinct $B,C \in \mathcal{P}$.
\end{enumerate}
In fact, it is enough to show that the set $\mathcal{G}$ defined in \eqref{eq:definitionGAGD} contains a perfect subset $\mathcal{P}$ which satisfies item \ref{item:rmk3perfect} (in fact, by the proof above, every $B \in \mathcal{G}$ belongs to $\mathcal{D}$ and satisfies item \ref{item:rmk2perfect}). 

To this aim, recall that $\mathcal{G}$ has full $\nu$-measure, i.e., $\nu(\mathcal{G})=1$. Now, $\nu$ is an atomless Borel probability measure on the Polish space $2^\omega$. 
By Ulam's theorem, $\nu$ is inner regular, hence there exists a compact set $\mathcal{K}\subseteq\mathcal{G}$ such that $\nu(\mathcal{K})>0$, see e.g. \cite[Theorem 23.2]{MR4226142}. 
Observe that $\mathcal{K}$ itself is a Polish space by \cite[Proposition 3.7 and Theorem 3.11]{K}. 
It follows by the Cantor--Bendixson theorem \cite[Theorem 6.4]{K} that $\mathcal{K}$ can be uniquely written as $\mathcal{K}=\mathcal{H}\cup \mathcal{C}$, with $\mathcal{H}$ perfect and $\mathcal{C}$ countable open. 
Since $\nu$ is atomless, $\nu(\mathcal{C})=0$, hence $\nu(\mathcal{H})=\nu(\mathcal{K})>0$, so that $\mathcal{H}$ is a nonempty perfect subset of $\mathcal{K}$.

Next, define
$$
\mathcal{E}:=\{(X,Y)\in 2^\omega\times 2^\omega:X\bigtriangleup Y \in \mathcal{Z}\}.
$$
and consider the continuous map $\Delta:2^\omega\times 2^\omega\to 2^\omega$ defined by $\Delta(X,Y):=X\bigtriangleup Y$. 
Since $\mathcal{E}=\Delta^{-1}[\mathcal{Z}]$, it follows that $\mathcal{E}$ is meager in $2^\omega\times 2^\omega$. 
This implies that $\mathcal{E}\cap(\mathcal{H}\times \mathcal{H})$ is meager in $\mathcal{H}\times \mathcal{H}$. We conclude by Mycielski's theorem, 
see e.g. \cite[Theorem 19.1]{K}, 
there exists a perfect set $\mathcal{P}\subseteq \mathcal{H}$ such that $(B,C)\notin \mathcal{E}$ for all distinct $B,C \in \mathcal{P}$, i.e., item \ref{item:rmk3perfect} holds. 
Taking into account that $\mathcal{P}\subseteq \mathcal{H} \subseteq \mathcal{K}\subseteq \mathcal{G}$, we conclude that $\mathcal{P}$ is a nonempty perfect set satisfying items \ref{item:rmk1perfect}-\ref{item:rmk3perfect}.
\end{rmk}

\medskip

\begin{proof}
[Proof of Corollary \ref{cor:extension1}]
In both cases, $\mathcal{A}$ is not maximal. Hence it is possible to pick $A \in \mathcal{D}$ such that $\mathcal{A}\cup \{A\}$ is $\mathsf{d}$-independent. Thanks to Proposition \ref{prop:notextension2}, there exists $B \in \mathcal{D}$ such that $\mathsf{d}(B)=s$ and $\mathcal{A}\cup \{B\}$ is $\mathsf{d}$-independent. 
\end{proof}

\medskip

Using a minor variation on the proof of Proposition \ref{prop:notextension2}\ref{item2extension}, we will need later a slight improvement of Corollary \ref{cor:extension1}\ref{item2extensioncc} in the case $s=\nicefrac{1}{2}$:
\begin{cor}\label{cor:extensionfinal4} 
Assume $\mathrm{cov}(\mathcal{N})=\mathfrak{c}$. Let $\mathcal{A}$ be a $\mathsf{d}$-independent family with $|\mathcal{A}|<\mathfrak{c}$ and pick $G\subseteq 2^\omega$ which is Lebesgue null. Then there exists $B \in \mathcal{D}\setminus G$ such that $\mathsf{d}(B)=\nicefrac{1}{2}$ and $\mathcal{A}\cup \{B\}$ is $\mathsf{d}$-independent\textup{.} \end{cor}
\begin{proof}
We will keep the same notation and strategy proof used in Proposition \ref{prop:notextension2}\ref{item2extension}. In place of the auxiliary product measure $\nu$ used there, we simply work with the usual product probability measure $\lambda$ on $2^\omega$. %, that is, the case in which each coordinate has probability $\nicefrac12$. 
Thus, for each $D\in\mathscr A$, there exists a set $\mathcal G_D\subseteq 2^\omega$ of full $\lambda$-measure such that 
$S\cap D\in\mathcal D$ and $\mathsf d(S\cap D)=\frac12\,\mathsf d(D)$ for all $S\in\mathcal G_D$. 
In particular, taking $D=\omega$, we obtain that every $S\in\mathcal G_\omega$ belongs to $\mathcal D$ and satisfies $\mathsf d(S)=\nicefrac12$. 
At this point, instead of \eqref{eq:definitionGAGD}, define
$$
\mathcal G:=(2^\omega\setminus G)\cap \bigcap_{D\in\mathscr A}\mathcal G_D.
$$
Since $|\mathscr A|=\kappa<\mathfrak c=\mathrm{cov}(\mathcal N)$ and  $\lambda(G)=\lambda(2^\omega\setminus\mathcal G_D)=0$ for all $D \in \mathscr{A}$, it follows that $\mathcal G$ is nonempty. 
By the same argument, each $B\in\mathcal G$ satisfies $B \in \mathcal{D}\setminus G$, $\mathsf{d}(B)=\nicefrac{1}{2}$, and $\mathcal{A}\cup \{B\}$ is $\mathsf{d}$-independent.
\end{proof}

\medskip

\begin{proof}
[Proof of Theorem \ref{thm:mainprescribedimage}]
Fix an enumeration $\{D_\alpha:\alpha<\mathfrak c\}$ of all sets $D \in \mathcal{D}$ with $\mathsf{d}(D) \in (0,1)$, and an enumeration $(p_\alpha:\alpha<\mathfrak c)$ with values in $S$ such that $\{p_\alpha:\alpha<\mathfrak c\}=S$. 
We construct by transfinite recursion an increasing chain
$\langle \mathcal A_\alpha:\alpha<\mathfrak c\rangle$ of families of subsets of $\omega$ such that for every $\alpha<\mathfrak c$:
\begin{enumerate}[label={\rm(\alph*)}]
\item \label{item:aclaim} $\mathcal A_\alpha$ is $\mathsf d$-independent  and $|\mathcal A_\alpha|\le |\alpha| <\mathfrak c$ whenever $\alpha>0$;
\item \label{item:bclaim} $\mathsf d[\mathcal A_\alpha]=\{p_\tau: \tau<\alpha \}\subseteq S$;
\item \label{item:cclaim} If $D_\alpha \notin \mathcal{A}_\alpha$ then $\mathcal A_{\alpha+1}\cup\{D_\alpha\}$ is not
$\mathsf d$-independent.
\end{enumerate}

Put $\mathcal A_0:=\emptyset$. At a limit ordinal $\beta<\mathfrak{c}$, define $\mathcal A_\beta:=\bigcup_{\alpha<\beta}\mathcal A_\alpha$. Since $\mathsf{d}$-independence is preserved under unions of increasing chains (as conditions \ref{defi:1indep}--\ref{defi:3indep} are finitary), then $\mathcal{A}_\beta$ is $\mathsf{d}$-independent. In addition, $|\mathcal A_\beta|\le |\beta|<\mathfrak c$ and $\mathsf{d}[\mathcal{A}_\beta]=\bigcup_{\alpha<\beta}\mathsf{d}[\mathcal{A}_\alpha]=\{p_\tau: \tau<\beta\}$. 

Now, fix $\alpha<\mathfrak{c}$ and suppose that $\mathcal{A}_\alpha$ has been defined. 
\begin{itemize}
\item If $D_\alpha \in \mathcal{A}_\alpha$, apply Corollary \ref{cor:extension1} to obtain a set $A_\alpha\in\mathcal D$ such that $\mathsf d(A_\alpha)=p_\alpha$ and $\mathcal A_\alpha\cup\{A_\alpha\}$ is $\mathsf d$-independent. Then set 
\begin{equation}\label{eq:defAalphaplusone}
\mathcal A_{\alpha+1}:=\mathcal A_\alpha\cup\{A_\alpha\}.
\end{equation}

\item If $D_\alpha \notin \mathcal{A}_\alpha$ and $\mathcal A_\alpha\cup\{D_\alpha\}$ is not $\mathsf d$-independent, apply again Corollary \ref{cor:extension1} to obtain a set $A_\alpha\in\mathcal D$ such that
$\mathsf d(A_\alpha)=p_\alpha$ and $\mathcal A_\alpha\cup\{A_\alpha\}$ is $\mathsf d$-independent. Define $\mathcal{A}_{\alpha+1}$ as in \eqref{eq:defAalphaplusone}. 
Then $\mathcal A_{\alpha+1}\cup\{D_\alpha\}$ is not $\mathsf d$-independent, since it contains $\mathcal A_\alpha\cup\{D_\alpha\}$. 

\item If $D_\alpha \notin \mathcal{A}_\alpha$ and $\mathcal A_\alpha\cup\{D_\alpha\}$ is $\mathsf d$-independent, apply Proposition \ref{prop:notextension2} to obtain $A_\alpha\in\mathcal D$ such that $\mathsf d(A_\alpha)=p_\alpha$, $\mathcal A_\alpha\cup\{A_\alpha\}$ is $\mathsf d$-independent, and $\{A_\alpha,D_\alpha\}$ is not $\mathsf{d}$-independent. 
Define $\mathcal{A}_{\alpha+1}$ as in \eqref{eq:defAalphaplusone}. Then $\mathcal A_{\alpha+1}\cup\{D_\alpha\}$ is not $\mathsf d$-independent, since it contains $\{A_\alpha,D_\alpha\}$. 
\end{itemize}
In all cases, $\mathcal A_{\alpha+1}$ is $\mathsf d$-independent, $|\mathcal A_{\alpha+1}|\le |\alpha+1|<\mathfrak c$, and
$\mathsf d[\mathcal A_{\alpha+1}]=\{p_\tau: \tau<\alpha+1\}\subseteq S$ since $\mathsf d(A_\alpha)=p_\alpha\in S$. 

Finally, set
$$
\mathcal A:=\bigcup_{\alpha<\mathfrak c}\mathcal A_\alpha.
$$
Then $\mathcal A$ is $\mathsf d$-independent, and  
we get by construction 
$$
\mathsf{d}[\mathcal{A}]
=\bigcup_{\alpha<\mathfrak{c}}\mathsf{d}[\mathcal{A}_\alpha]
=\{p_\alpha: \alpha< \mathfrak{c}\}
=S.
$$
To conclude the proof, we need to show that $\mathcal{A}$ is maximal. In fact, pick $B\in\mathcal D\setminus \mathcal{A}$ with $\mathsf d(B)\in(0,1)$. Then $B=D_\alpha$ for some $\alpha<\mathfrak c$. Since $B\notin \mathcal{A}$, then $B=D_\alpha \notin \mathcal{A}_\alpha$. By construction, $\mathcal A_{\alpha+1}\cup\{D_\alpha\}$ is not $\mathsf d$-independent, and hence neither is
$\mathcal A\cup\{B\}$. Therefore no set in $\mathcal D\setminus \mathcal{A}$ with asymptotic density in $(0,1)$ can be adjoined to $\mathcal A$ to preserve $\mathsf d$-independence. 
\end{proof}

\medskip

\begin{proof}
[Proof of Theorem \ref{thm:imagefieldsfidtsinct}]
Let $T$ be a transcendental basis of $\mathbb{R}$ over $\mathbb{Q}$ contained in $(0,1)$,\footnote{By the existence of transcendence bases, choose a maximal algebraically independent subset $T \subseteq S=\left\{\frac{1}{x^2+2}:x\in\mathbb{R}\right\}\subseteq(0,1)$.
Since every \(x\in\mathbb{R}\) is algebraic over $\mathbb{Q}\!\left(\frac{1}{x^2+2}\right)$, the field $\mathbb{R}$ is algebraic over $\mathbb{Q}(S)$, hence over $\mathbb{Q}(T)$. Therefore $T$ is a transcendence basis of $\mathbb{R}$ over $\mathbb{Q}$.} that is, a maximal algebraically independent set, see e.g. \cite[Section 8.12]{MR1009787}. It is folklore that $|T|=\mathfrak{c}$. 
%, hence it is possible to pick an enumeration $(t_\alpha: \alpha<\mathfrak{c})$. 
%https://math.stackexchange.com/questions/3957849/basis-and-transcendental-basis
Define an equivalence relation $\mathsf{E}$ on $\mathcal P(T)$ by setting $S\mathsf{E} R$ if and only if $|S\bigtriangleup R|\le \omega$. 
For every $S\subseteq T$, the equivalence class $[S]_\mathsf{E}$ of $S$ has cardinality at most $\mathfrak c$, because
$$
|[S]_{\mathsf{E}}|=|\{S\bigtriangleup C: C\in [T]^{\le\omega}\}| \le \big|[T]^{\le \omega}\big|=\mathfrak c.
$$
It follows that the quotient $\mathcal P(T)/\mathsf{E}$ has cardinality $2^{\mathfrak c}$.
Therefore we can fix a family 
$
\{S_\alpha:\alpha<2^{\mathfrak c}\}\subseteq \mathcal P(T)\setminus\{\emptyset\}
$ 
such that $S_\alpha\bigtriangleup S_\beta$ is uncountable for all distinct $\alpha,\beta<2^{\mathfrak c}$. 

Now, for each $\alpha<2^{\mathfrak c}$, apply Theorem \ref{thm:jon} to the set $S_\alpha\subseteq (0,1)$.
Thus there exist a maximal $\mathsf d$-independent family $\mathcal{A}_\alpha\in \mathscr M$ of cardinality $\mathfrak{c}$ and an at most countable subfamily
$\mathcal{B}_\alpha\subseteq \mathcal{A}_\alpha$ such that
\begin{equation}\label{eq:thmjoneta}
\mathsf d[\mathcal{A}_\alpha\setminus \mathcal{B}_\alpha]=S_\alpha
\qquad\text{ and }\qquad
\mathsf d[\mathcal{B}_\alpha]^\prime\subseteq \{0,1\}.
\end{equation}
Define $C_\alpha:=\mathsf d[\mathcal{B}_\alpha]$. 
Since $\mathcal{B}_\alpha$ is at most countable, then $C_\alpha$ is countable as well.

For each $x\in C_\alpha$, since $T$ is a transcendence basis of $\mathbb R$ over $\mathbb Q$, the real number $x$ is algebraic over $\mathbb Q(T)$.
Hence there exists a finite set $F_x\subseteq T$ such that $x$ is algebraic over $\mathbb Q(F_x)$.
Define
$$
U_\alpha:=\bigcup_{x\in C_\alpha}F_x.
$$
Since $C_\alpha$ is countable and each $F_x$ is finite, the set $U_\alpha\subseteq T$ is countable.
\begin{claim}\label{eq:keycodingfields}
$S_\alpha
\subseteq
T\cap \mathsf d[\langle \mathcal{A}_\alpha\rangle]
\subseteq
S_\alpha\cup U_\alpha$ for each $\alpha<2^{\mathfrak{c}}$. 
\end{claim}
\begin{proof}
Fix $\alpha<2^{\mathfrak{c}}$. To prove the left inclusion, fix $t\in S_\alpha$.
By \eqref{eq:thmjoneta}, there exists $A\in \mathcal{A}_\alpha\setminus \mathcal{B}_\alpha$ such that
$\mathsf d(A)=t$.
Since $A\in \mathcal{A}_\alpha\subseteq \langle \mathcal{A}_\alpha\rangle$, it follows that
$t\in \mathsf d[\langle \mathcal{A}_\alpha\rangle]$.
Moreover $t\in T$ because $S_\alpha\subseteq T$.
Thus $S_\alpha\subseteq T\cap \mathsf d[\langle \mathcal{A}_\alpha\rangle]$. 

To prove the right inclusion, pick $
t\in T\cap \mathsf d[\langle \mathcal{A}_\alpha\rangle]$. 
Then there exists $X\in \langle \mathcal{A}_\alpha\rangle$ such that $\mathsf d(X)=t$.
Since $X$ belongs to the field generated by $\mathcal{A}_\alpha$, there are a finite set
$\mathcal{F}\subseteq \mathcal{A}_\alpha$ and a set $E\subseteq 2^{\mathcal{F}}$ such that
$
X=\bigcup_{\sigma\in E}\bigcap_{A\in\mathcal{F}} A^{\sigma(A)}$, 
where the union is disjoint.
Hence, using \eqref{eq:d3sharp}, we obtain
$$
\mathsf d(X)
=
\sum_{\sigma\in E}
\prod_{A\in\mathcal{F}}\mathsf d\big(A^{\sigma(A)}\big).
$$
It follows that $\mathsf d(X)$ belongs to the subring of $\mathbb R$ generated by $\{\mathsf d(A):A\in\mathcal{F}\}$. In particular, $
t=\mathsf d(X)\in \mathbb Q\big(\mathsf d[\mathcal{F}]\big)$. 
Now
$$
\mathsf d[\mathcal{F}]
\subseteq
\mathsf d[\mathcal{A}_\alpha]
=
\mathsf d[\mathcal{A}_\alpha\setminus \mathcal{B}_\alpha]\cup \mathsf d[\mathcal{B}_\alpha]
=
S_\alpha\cup C_\alpha
$$
by \eqref{eq:thmjoneta}. Therefore $t\in \mathbb Q(S_\alpha\cup C_\alpha)$. Since every element of $C_\alpha$ is algebraic over $\mathbb Q(U_\alpha)$ by the definition of $U_\alpha$, it follows that
$t$ is algebraic over $\mathbb Q(S_\alpha\cup U_\alpha)$. 
Now, assume towards a contradiction that $t\in T\setminus (S_\alpha\cup U_\alpha)$.
Since $T$ is algebraically independent over $\mathbb Q$, the element $t$ is transcendental over
$\mathbb Q(S_\alpha\cup U_\alpha)$.
This contradicts the previous observation that $t$ is algebraic over $\mathbb Q(S_\alpha\cup U_\alpha)$, and proves the right inclusion. 
\end{proof}

To conclude the proof, suppose towards a contradiction that there exist distinct $\alpha,\beta<2^{\mathfrak{c}}$ such that
$
\mathsf d[\langle \mathcal{A}_\alpha\rangle]
=
\mathsf d[\langle \mathcal{A}_\beta\rangle],
$
and therefore
$$
T\cap \mathsf d[\langle \mathcal{A}_\alpha\rangle]
=
T\cap \mathsf d[\langle \mathcal{A}_\beta\rangle].
$$
Using Claim \ref{eq:keycodingfields}, we obtain that 
$$
S_\alpha\bigtriangleup S_\beta
=
(S_\alpha\setminus S_\beta)\cup (S_\beta\setminus S_\alpha)
\subseteq U_\alpha \cup U_\beta. 
$$
Since both $U_\alpha$ and $U_\beta$ are countable, this would imply that $S_\alpha \mathsf{E} S_\beta$, proving the desired contradiction. 
\end{proof}

\medskip

\begin{rmk}\label{rmk:dense}
  Let $\mathcal{A}$ be an uncountable $\mathsf{d}$-independent family (which is the case if $\mathcal{A}$ is maximal, thanks to Corollary \ref{cor:ordercardinals}). Then $\mathsf{d}[\langle \mathcal{A}\rangle]$ is dense in $[0,1]$. 
  
  In fact, fix a nonempty open set $U\subseteq (0,1)$. Since $\mathcal{A}$ is uncountable, there exists an integer $m\ge 2$ such that 
  $$
  \mathcal{B}:=\{A \in \mathcal{A}: \nicefrac{1}{m} \le \mathsf{d}(A) \le 1-\nicefrac{1}{m}\}
  $$
  is infinite. Let $(A_n: n \in\omega)$ be an enumeration of a countably infinite subset of $\mathcal{B}$. Then 
  $
  \mathsf{d}[\langle \{A_i: i \in n\}\rangle]\subseteq 
  \mathsf{d}[\langle \mathcal{A}\rangle]
  $ 
  for each $n \in \omega$. 
  Now, notice that $\mathcal{Q}_n:=\{\bigcap_{i \in n}A_i^{\sigma(i)}: \sigma \in 2^n\}\subseteq \mathcal{D}$ is a partition of $\omega$ and $q_n:=\max\{\mathsf{d}(S): S \in \mathcal{Q}_n\} \le (1-\nicefrac{1}{m})^n$. Hence $\lim_n q_n=0$. Therefore there exists $n \in \omega$ and a nonempty $F\subseteq 2^n$ such that $A:=\bigcup_{\sigma \in F}\bigcap_{i \in n}A_i^{\sigma(i)} \in \mathcal{D}$ and $\mathsf{d}\left(A\right) \in U$. 

  On the other hand, 
  %the analogue claim fails 
  $\mathsf{d}[\langle \mathcal{A}\rangle]$ might not be dense in $[0,1]$ if $\mathcal{A}$ is countably infinite. 
  %if $\mathcal{A}$ is countable. 
  To this aim, pick a sequence $(p_n:n\in\omega)$ in $(0,1)$ such that 
$p:=\prod_{n}p_n>\nicefrac12$. By Proposition \ref{prop:continuum_indep}, there exists a $\mathsf{d}$-independent family $\mathcal{A}:=\{A_n:n\in\omega\}$ 
such that $\mathsf{d}(A_n)=p_n$ for all $n \in \omega$. Define $B_n:=\bigcap_{i \in n}A_i$ for all $n \in \omega$ and note that $(\mathsf{d}(B_n): n \in \omega)$ is a decreasing sequence with limit $p>\nicefrac{1}{2}$. Now, pick $A \in \langle \mathcal{A}\rangle$. Then there exists $n \in \omega$ such that $A \in \langle \{A_i: i \in n\}\rangle$. If $B_n\subseteq A$ then $\mathsf{d}(A)\ge \mathsf{d}(B_n)>p$. In the opposite case, we have $B_n \cap A=\emptyset$, hence $\mathsf{d}(A)\le \mathsf{d}(\omega\setminus B_n)<1-p$. Therefore $\mathsf{d}[\langle \mathcal{A}\rangle]\cap (1-p,p)=\emptyset$. 
\end{rmk}

\medskip

\begin{proof}
[Proof of Corollary \ref{cor:quantity2c}]
    On the one hand, it is clear that $|\{\mathsf{d}[\langle \mathcal{A}\rangle]: \mathcal{A} \in \mathscr{M}\}|\le |\{\langle \mathcal{A}\rangle: \mathcal{A} \in \mathscr{M}\}|\le |\mathscr{M}|\le 2^\mathfrak{c}$. On the other hand, $2^\mathfrak{c}\le |\{\mathsf{d}[\langle \mathcal{A}\rangle]: \mathcal{A} \in \mathscr{M}\}|$ by Theorem \ref{thm:imagefieldsfidtsinct}. 
\end{proof}

\medskip

\begin{proof}
[Proof of Corollary \ref{cor:notanalytic}] 
    It follows by 
    %Theorem \ref{thm:imagefieldsfidtsinct} 
    Corollary \ref{cor:quantity2c} 
    and the known fact that there are $\mathfrak{c}$ analytic subsets of $2^\omega$.\footnote{Since every analytic subset of $2^\omega$ is the projection of a closed subset of $2^\omega\times \omega^\omega$, and an uncountable Polish space has only $\mathfrak c$ Borel (hence closed) subsets, see e.g. \cite[Chapter 14]{K} and \cite{MR4003698}, it follows that there are at most $\mathfrak c$ analytic subsets of $2^\omega$; the reverse inequality is trivial because $2^\omega$ has $\mathfrak c$ singletons.}
\end{proof}

\medskip

\begin{rmk}
Looking for another strengthening of the existence of nonanalytic maximal $\mathsf{d}$-independent families, it is worth remarking that Larson, Neeman, and Shelah proved in \cite{MR2640071} that it is consistent in $\mathsf{ZFC}$ that there are only $\mathfrak{c}$ universally measurable sets in $2^\omega$ (here, we recall that a subset of $2^\omega$ is universally measurable if it is $\mu$-measurable for any $\sigma$-finite Borel measure $\mu$ on $2^\omega$, see e.g. \cite[p. 155]{K}). 
%; cf. \cite[pp. 576-578]{MR716618} for a related weaker claim. 
Hence, with the same counting argument in Corollary \ref{cor:notanalytic}, it is consistent that there exists a maximal $\mathsf{d}$-independent family $\mathcal{A}$ such that neither $\mathcal{A}$ nor $\langle \mathcal{A}\rangle$ nor $\mathsf{d}[\langle \mathcal{A}\rangle]$ is universally measurable. 
\end{rmk}

Since there are $2^{\mathfrak{c}}$ Lebesgue measurable subsets of $2^\omega$ which have also the Baire property, it is evident that the strategy used in the proof of Corollary \ref{cor:notanalytic} cannot be adapted for Theorem \ref{thm:notbaire} and Theorem \ref{thm:notmeasurable}. 

\medskip

\begin{proof}
[Proof of Theorem \ref{thm:notbaire}] 
For each $n \in \omega$ and $s \in 2^n$, define $Y_n:=\{n\}\times \mathcal{P}(2^n)$ and $X^n_s:=\{(n,A) \in Y_n: s \in A\}$. Taking inspiration from the eighth proof of \cite[Theorem 3.4]{Geschke}, it follows that 
$$
\bigcap_{s \in S}X^n_s \cap \left(Y_n\setminus \bigcup_{t \in T}X^n_t\right)\neq \emptyset
$$
for all disjoint $S,T\subseteq 2^n$ (in fact, it contains $(n,S)$). Now, set $Y:=\bigcup_n Y_n$, pick a bijection $h: Y \to \omega$, and define $X_\sigma:=\bigcup_n h[X^n_{\sigma \upharpoonright n}]$ for each $\sigma \in 2^\omega$. Then $\{X_\sigma:\sigma\in 2^\omega\}\subseteq \mathcal P(\omega)$ is an independent family. 
Moreover, by construction, the map $\sigma\mapsto X_\sigma$ is continuous, since membership on the block $Y_n$ depends only on $\sigma\restriction n$, and injective, since distinct finite strings in $2^n$ yield distinct sets. 
%By \cite[Theorem 3.4]{Geschke}---more precisely, by the block construction in the eighth proof---there exist a partition $\{Y_n: n \in \omega\}$ of $\omega$ into nonempty finite sets and a family $\{X_\sigma:\sigma\in 2^\omega\}\subseteq \mathcal P(\omega)$ which is independent. 
%Moreover, in that construction the map $\sigma\mapsto X_\sigma$ is continuous, since membership on the block $Y_n$ depends only on $\sigma\restriction n$, and injective, since distinct finite strings in $2^n$ yield distinct sets. 
Hence its image is homeomorphic to $2^\omega$, and therefore is a perfect subset of $2^\omega$. 

Let us rename the above perfect independent family by $\mathcal{B}=\{B_\alpha: \alpha<\mathfrak{c}\}$. 
Applying the technique in Remark
\ref{rmk:relationshipindependent}, we construct a $\mathsf{d}$-independent family $\mathcal{A}=\{A_\alpha: \alpha<\mathfrak{c}\}$ such that $\mathsf{d}[\mathcal{A}]=\{\nicefrac{1}{2}\}$. In addition, the map $2^\omega \to 2^\omega$ defined by $B_\alpha\mapsto A_\alpha$ (or, more precisely, $B\mapsto A_B$ with its obvious meaning) is continuous and injective. Hence $\mathcal{A}$ is also a perfect subset of $2^\omega$. Fix also a Bernstein subset $\mathcal{R}\subseteq \mathcal{A}$, cf. \cite[Example 8.24]{K}. 
By Zorn's lemma, pick a maximal $\mathsf d$-independent family $\mathcal M\in\mathscr M$ containing $\mathcal{A}$, and define the family 
$$
\mathcal{F}:=(\mathcal{M}\setminus \mathcal{A}) \cup \mathcal{R} \cup \{\omega\setminus S: S \in \mathcal{A}\setminus \mathcal{R}\}. 
$$
By the observation preceding Theorem \ref{thm:imagefieldsfidtsinct},
$\mathcal{F}$ is still a maximal $\mathsf d$-independent family. Since $\mathcal A\cap\mathcal F=\mathcal R$, 
if $\mathcal F$ had the Baire property in $2^\omega$, then
$\mathcal R$ would have the relative Baire property in $\mathcal A$.
But this is impossible, because a Bernstein subset of a perfect Polish
space does not have the Baire property.
Therefore $\mathcal F$ is a maximal $\mathsf d$-independent family without the Baire property.
\end{proof}

\medskip

\begin{proof}
   [Proof of Theorem \ref{thm:notmeasurable}]
   Fix an enumeration $\{D_\alpha:\alpha<\mathfrak c\}$ of all sets $D \in \mathcal{D}$ with $\mathsf{d}(D) \in (0,1)$, and an enumeration $\{G_\alpha:\alpha<\mathfrak c\}$ of all Lebesgue null $G_\delta$ subsets of $2^\omega$. 
We construct by transfinite recursion an increasing chain
$\langle \mathcal A_\alpha:\alpha<\mathfrak c\rangle$ of families of subsets of $\omega$ such that for every $\alpha<\mathfrak c$:
\begin{enumerate}[label={\rm(\alph*)}]
\item \label{item:ameasclaim} $\mathcal A_\alpha$ is $\mathsf d$-independent and $|\mathcal A_\alpha|\le |\alpha|+\omega<\mathfrak c$ whenever $\alpha>0$\textup{;}
\item \label{item:bmeasclaim} $\mathcal A_{\alpha+1}\setminus G_\alpha\neq\emptyset$\textup{;}
\item \label{item:cmeasclaim} If $D_\alpha\notin \mathcal A_{\alpha+1}$ then $\mathcal A_{\alpha+1}\cup\{D_\alpha\}$ is not $\mathsf d$-independent\textup{.}
\end{enumerate}

Put $\mathcal A_0:=\emptyset$. At a limit ordinal $\beta<\mathfrak{c}$, define $\mathcal A_\beta:=\bigcup_{\alpha<\beta}\mathcal A_\alpha$. Since $\mathsf{d}$-independence is preserved under unions of increasing chains (as conditions \ref{defi:1indep}--\ref{defi:3indep} are finitary), then $\mathcal{A}_\beta$ is $\mathsf{d}$-independent. In addition, $|\mathcal A_\beta|\le |\beta|+\omega<\mathfrak c$.

Now, fix $\alpha<\mathfrak{c}$ and suppose that $\mathcal{A}_\alpha$ has been defined. Apply Corollary \ref{cor:extensionfinal4} to $\mathcal A_\alpha$ and $G_\alpha$ to obtain a set $X_\alpha\in\mathcal D\setminus G_\alpha$ such that $\mathsf d(X_\alpha)=\nicefrac12$ and  $\mathcal A_\alpha\cup\{X_\alpha\}$ is $\mathsf{d}$-independent. Now, consider the following cases: 
\begin{itemize}
\item If $D_\alpha \in \mathcal A_\alpha\cup\{X_\alpha\}$, then set 
\begin{equation}\label{eq:definitionfamilyAalpha}
\mathcal A_{\alpha+1}:=\mathcal A_\alpha\cup\{X_\alpha\}.
\end{equation}

\item If $D_\alpha \notin \mathcal A_\alpha\cup\{X_\alpha\}$ and $\mathcal A_\alpha\cup\{X_\alpha,D_\alpha\}$ is not $\mathsf d$-independent, then define again $\mathcal A_{\alpha+1}$ as in \eqref{eq:definitionfamilyAalpha}. 

\item If $D_\alpha \notin \mathcal A_\alpha\cup\{X_\alpha\}$ and $\mathcal A_\alpha\cup\{X_\alpha,D_\alpha\}$ is $\mathsf d$-independent, apply Proposition \ref{prop:notextension2}\ref{item2extension} to obtain $Y_\alpha\in\mathcal D$ such that $
\mathsf d(Y_\alpha)=\frac12$, $\mathcal A_\alpha\cup\{X_\alpha,Y_\alpha\}$ is $\mathsf d$-independent, and $\{Y_\alpha,D_\alpha\}$ is not $\mathsf d$-independent. 
Then set
$$
\mathcal A_{\alpha+1}:=\mathcal A_\alpha\cup\{X_\alpha, Y_\alpha\}.
$$
\end{itemize}

In all cases, $\mathcal A_{\alpha+1}$ is $\mathsf d$-independent and $|\mathcal A_{\alpha+1}|\le |\alpha+1|+\omega<\mathfrak c$.
In addition, by construction we have $X_\alpha\in \mathcal A_{\alpha+1}\setminus G_\alpha$, so item \ref{item:bmeasclaim} holds.
Lastly, item \ref{item:cmeasclaim} is immediate in the first two cases, while in the third case it follows because $\mathcal A_{\alpha+1}\cup\{D_\alpha\}$ contains $\{Y_\alpha,D_\alpha\}$, which is not $\mathsf{d}$-independent. 

Finally, set
$$
\mathcal A:=\bigcup_{\alpha<\mathfrak c}\mathcal A_\alpha.
$$
Then $\mathcal A$ is $\mathsf d$-independent. 
Now, we show that $\mathcal{A}$ is maximal. In fact, pick $B\in\mathcal D\setminus \mathcal A$ with $\mathsf d(B)\in(0,1)$. Then $B=D_\alpha$ for some $\alpha<\mathfrak c$. Since $B\notin \mathcal A$, we have $B=D_\alpha\notin \mathcal A_{\alpha+1}$. By construction, $\mathcal A_{\alpha+1}\cup\{D_\alpha\}$ is not $\mathsf d$-independent, and hence neither is
$\mathcal A\cup\{B\}$. Therefore $\mathcal A\in\mathscr M$. 

To conclude the proof, it remains to show that $\mathcal A$ is not Lebesgue measurable. To this aim, assume for the sake of contradiction that $\mathcal{A}$ is Lebesgue measurable. 
For each finite $F\subseteq \omega$, define the measure-preserving homeomorphism $\tau_F: 2^\omega \to 2^\omega$ by $\tau_F(A):=A\bigtriangleup F$. Hence each $\mathcal{A}^{F}:=\tau_F[\mathcal{A}]$ is Lebesgue measurable and $\lambda(\mathcal{A}^{F})=\lambda(\mathcal{A})$. Since $\mathcal{A}$ is, in particular, an independent family, it follows that the families in $\{\mathcal{A}^F: F\in [\omega]^{<\omega}\}$ are pairwise disjoint. Since $1=\lambda(2^\omega)\ge \sum_{F\in [\omega]^{<\omega}}\lambda(\mathcal{A}^F)$, we obtain that $\lambda(\mathcal{A})=0$. 
Thus, for every $n\in\omega$ it is possible to pick an open set $U_n\subseteq 2^\omega$ such that
$\mathcal A\subseteq U_n$ and $\lambda(U_n)<2^{-n}$. 
Define
\[
G:=\bigcap_{n\in\omega}U_n.
\]
Then $G$ is a Lebesgue null $G_\delta$ set and $\mathcal A\subseteq G$.
By the choice of the enumeration $(G_\alpha:\alpha<\mathfrak c)$, we have $G=G_\alpha$ for some $\alpha<\mathfrak c$.
However $\mathcal A_{\alpha+1}\setminus G_\alpha\neq\emptyset$  by item \ref{item:bmeasclaim}.
Since $\mathcal A_{\alpha+1}\subseteq \mathcal A\subseteq G_\alpha$, this provides the desired contradiction.
Therefore $\mathcal A$ is not Lebesgue measurable.
\end{proof}

%\nocite{*}
\bibliographystyle{amsplain}
\bibliography{indep}

\end{document}